\documentclass{article}
\usepackage[utf8]{inputenc}
\usepackage[T1]{fontenc}
\usepackage{amsmath,latexsym}
\usepackage{amssymb}
\usepackage[a4paper]{geometry}
\usepackage{aeguill}                 
\usepackage{graphicx}
\usepackage{amsmath,amsfonts,amssymb,amsthm}
\usepackage{pstricks} 
\usepackage{epstopdf}
\usepackage{srcltx} 
\usepackage{hyperref}

\usepackage[english]{babel}

\newcommand\Out{\mathrm{Out}}
\newcommand\Aut{\mathrm{Aut}}
\newcommand\BP{\mathrm{PB}}
\newcommand\Stab{\mathrm{Stab}}
\newcommand\calz{\mathcal{Z}}
\newcommand\calo{\mathcal{O}}
\newcommand\calv{\mathcal{V}}
\newcommand\calm{\mathcal{M}}
\newcommand\calf{\mathcal{F}}

\newcommand\card{\mathrm{card}}

\newcommand\N{\mathbb{N}}

\newtheorem{theo}{Theorem}[section]
\newtheorem{lemma}[theo]{Lemma}

\newtheorem{prop}[theo]{Proposition}

\newtheorem{intro}{Theorem}
\newtheorem{introth}[intro]{Theorem}
\newtheorem{introcor}[intro]{Corollary}
\newtheorem{introquestion}[intro]{Question}

\theoremstyle{remark}
\newtheorem{Remk}[theo]{Remark}

\providecommand{\keywords}[1]
{
  \small	
  \textbf{\textit{Keywords---}} #1
}
\providecommand{\MSCcodes}[1]
{
  \small	
  \textbf{\textit{MSC---}} #1
}

\title{On the homology growth and the $\ell^2$-Betti numbers of $\Out(W_n)$}

\author{Damien Gaboriau, Yassine Guerch and Camille Horbez}
\date{\today}

\begin{document}

\maketitle

\begin{abstract}
     Let $n\ge 3$, and let $\Out(W_n)$ be the outer automorphism group of a free Coxeter group $W_n$ of rank $n$. We study the growth of the dimension of the homology groups (with coefficients in any field $\mathbb{K}$) along Farber sequences of finite-index subgroups of $\Out(W_n)$. We show that, in all degrees up to $\lfloor\frac{n}{2}\rfloor-1$, these Betti numbers grow sublinearly in the index of the subgroup. When $\mathbb{K}=\mathbb{Q}$, through Lück's approximation theorem, this implies that all $\ell^2$-Betti numbers of $\Out(W_n)$ vanish up to degree $\lfloor\frac{n}{2}\rfloor-1$. In contrast, in top dimension equal to $n-2$, an argument of Gaboriau and Noûs implies that the $\ell^2$-Betti number does not vanish.
       We also prove that the torsion growth of the integral homology is sublinear. Our proof of these results relies on a recent method introduced by Abért, Bergeron, Fr{\k{a}}czyk and Gaboriau. A key ingredient is to show that a version of the complex of partial bases of $W_n$ has the homotopy type of a bouquet of spheres of dimension $\lfloor\frac{n}{2}\rfloor-1$.
\end{abstract}

\MSCcodes{57M07,20J05,20F28,20E08,20E26,06A07}


\section{Introduction}

The homology groups of a Riemannian manifold or CW-complex $X$ are among the most studied and powerful invariants to understand its topology. Over time, the focus has shifted from the space $X$ to its fundamental group, and group homology has become an important topic on its own. However, a precise computation of the rational homology of a group is often a much complicated task; for example these numbers have an erratic behaviour with respect to simple group-theoretic operations like passing to a finite-index subgroup. To smoothen this behaviour, instead of considering the homology of the group $G$ itself, one can study the growth of the $\mathbb{Q}$-dimension of the homology over a suitable sequence of finite-index subgroups of $G$, normalized by the index. By Lück's approximation theorem \cite{Luc}, when $G$ is of type $F_{\infty}$, these numbers converge to the $\ell^2$-Betti numbers of $G$, as defined by Atiyah \cite{Ati} and Cheeger-Gromov \cite{CG} using $\ell^2$-chains and the  notion of von Neumann dimension.

In recent work, Abért, Bergeron, Fr{\k{a}}czyk and the first named author~\cite{ABFG} have developed new tools to compute the homology growth with coefficients in any field that apply to many families of groups and many degrees. This also enables to get some control on the growth of the torsion part of the homology with $\mathbb{Z}$ coefficients. They obtained results regarding the homology growth and homology torsion growth of $\mathrm{SL}_d(\mathbb{Z})$, mapping class groups of finite-type surfaces, and many Artin groups.

In the present paper, we study these questions for the group $\Out(W_n)$ of outer automorphisms of a free Coxeter group $W_n=\mathbb{Z}/2\mathbb{Z}\ast\dots\ast\mathbb{Z}/2\mathbb{Z}$ (the free product of $n$ cyclic factors of order~$2$). This group shares several rigidity properties with $\mathrm{SL}_d(\mathbb Z)$, the mapping class groups and $\Out(F_N)$, see e.g.\ \cite{Gue3}. Our main theorem is the following.

\begin{introth}\label{theointro:main}
For every $n\ge 3$, every Farber sequence $(\Gamma_k)_{k \in \mathbb N}$ of $\Out(W_n)$, every coefficient field $\mathbb{K}$ and every $0\leq j\leq \lfloor {\frac{n}{2}}\rfloor-1$ we have:
\[\lim_{k\to \infty} \frac{\dim_{\mathbb{K}} H_j (\Gamma_k, \mathbb{K})}{[\Gamma : \Gamma_k]} = 0 \quad \mbox{and} \quad \lim_{k\to \infty} \frac{\log | H_j(\Gamma_k,\mathbb Z)_{\rm tors}|}{[\Gamma:\Gamma_k]}=0.\]
\end{introth}

We refer to \cite[Definition~10.1]{ABFG} for the definition of a Farber sequence. Examples include decreasing sequences of finite-index normal subgroups with trivial intersection. Notice that $\Out(W_n)$ is residually finite \cite[Theorem~1.5]{MO}, so it admits Farber sequences.

In view of Lück's approximation theorem \cite{Luc} and $\Out(W_n)$ being of type $F_\infty$ (see Section~\ref{sect:CRP for stab and concl}), we obtain the following consequence regarding the $\ell^2$-Betti numbers of $\Out(W_n)$.

\begin{introcor}\label{corintro:main}
For every $n\geq 3$ and every $0\leq j\leq \lfloor {\frac{n}{2}}\rfloor-1$, one has $\beta^{(2)}_j(\Out(W_n))=0$.
\end{introcor}

To complement Corollary~\ref{corintro:main}, let us mention that on the other hand, the $\ell^2$-Betti number in top dimension (equal to the virtual cohomological dimension of $\Out(W_n)$, i.e.\ $n-2$, see~\cite[Corollary~10.2]{KV}), does not vanish. This was essentially established by Gaboriau and Noûs \cite[Theorem~1.1]{GN}: they proved it for $\Out(F_N)$, but the same argument applies to $\Out(W_n)$, see Section~\ref{sec:top-dimension} of the present paper.

\begin{introth}
For every $n\ge 3$, one has $\beta_{n-2}^{(2)}(\Out(W_n))>0$.
\end{introth}

Notice that besides these results, essentially nothing is known about the homology of $\Out(W_n)$.

In the case of $\Out(F_N)$, Gaboriau and Noûs asked whether all $\ell^2$-Betti numbers up to dimension $2n-4$ (i.e.\ except for the top-dimensional one) vanish \cite[Question~1.2]{GN}. This would match a theorem of Borinsky and Vogtmann \cite[Theorem~A]{BV} asserting that the Euler characteristic of $\Out(F_N)$ is negative for $N\ge 2$. We raise the analogous questions for $\Out(W_n)$.

\begin{introquestion}
Do all $\ell^2$-Betti numbers of $\Out(W_n)$ vanish up to dimension $n-3$? 
\\
Does the sign of the Euler characteristic of $\Out(W_n)$ alternate with the parity of $n$?
\end{introquestion}

Let us now say a word about our proof of Theorem~\ref{theointro:main}. This is based on the \emph{cheap $\alpha$-rebuilding property} introduced by Abért-Bergeron-Fr{\k{a}}czyk-Gaboriau in \cite[Definition~10.6]{ABFG}, from which the conclusions of Theorem~\ref{theointro:main} follow 
in degree $j$ up to $\alpha$. The advantage of this property is that it allows for inductive arguments. The starting point is that infinite virtually abelian finitely generated groups have the cheap rebuilding property. And if a  residually finite group $G$ acts cellularly, cocompactly, without cell inversions on a CW-complex $\Omega$ which is $(\alpha-1)$-connected, and in such a way that stabilizers of cells of dimension $j\le\alpha$ all have the cheap $(\alpha-j)$-rebuilding property, then $G$ itself has the cheap $\alpha$-rebuilding property (see Theorem~\ref{Th: 10.10 of [ABFG]} below, recording \cite[Theorem~10.9]{ABFG}).

Let us now describe the $\Out(W_n)$-simplicial complex $\Omega$ to which we apply this criterion. A \emph{partial $W_2$-basis} of $W_n$ is a nonempty finite set $\{[A_1],\dots,[A_k]\}$ of conjugacy classes of subgroups of $W_n$ that are all isomorphic to the infinite dihedral group $W_2$, and such that there exist representatives $A_i$ of $[A_i]$ and a subgroup $B\subseteq W_n$ (isomorphic to $W_{n-2k}$) such that $W_n=A_1\ast\dots\ast A_k\ast B$. The \emph{complex of partial $W_2$-bases} $|\BP_n|$ is the geometric realization of the partially ordered set of all partial $W_2$-bases of $W_n$. In fact, we only work with a subcomplex of $|\BP_n|$, defined as follows. We fix once and for all a finite $n$-tuple  $X^*=(x_1^*,\dots,x_n^*)$ of order $2$ elements forming a basis of $W_n$. We consider the subcomplex $|\BP_n^*|$ spanned by all partial $W_2$-bases that only contain conjugacy classes of free factors of the form $\langle gx_{2i-1}^*g^{-1},hx_{2i}^*h^{-1}\rangle$ for some $g,h\in W_n$ and some $i\in\{1,\dots,\lfloor\frac{n}{2}\rfloor\}$. The subcomplex $\Omega:=|\BP_n^*|$ is preserved by a finite-index subgroup $\Out^0(W_n)$ of $\Out(W_n)$. 

The stabilizers for the action of $\Out(W_n)$ on $|\BP_n|$, or of $\Out^0(W_n)$ on $|\BP^*_n|$, are well-understood (see Section~\ref{sect:CRP for stab and concl}).
Notably, given a cell $\sigma$ of $|\BP_n^*|$, its stabilizer $G_\sigma$ in $\Out^0(W_n)$ has the cheap $\beta$-rebuilding property for every $\beta\in \N$.
This is because $G_\sigma$ stabilizes the conjugacy class of some system of virtually cyclic free factors. So $G_\sigma$ acts on the spine of the corresponding \emph{relative Outer space} with infinite virtually abelian stabilizers, and this spine was proved to be contractible by Guirardel and Levitt \cite[Corollary~4.4]{GL}. Thus \cite[Theorem~10.9]{ABFG} also applies to the action of $G_\sigma$ on that spine.

Our main technical contribution in the present work is therefore to compute the homotopy type of the complex $|\BP_n^*|$.

\begin{introth}\label{theointro:bp}
For every $n\ge 3$, the complex $|\BP_n^*|$ has the homotopy type of a wedge of spheres of dimension $\lfloor\frac{n}{2}\rfloor-1$.
\end{introth}

This is similar to results of Solomon-Tits~\cite{Solomon-1969} for the rational Tits building of $\mathrm{SL}_d(\mathbb Z)$ and of  Harer~\cite{Harer} for the curve complex for the mapping class groups.

There are several analogues of these complexes for $\Out(F_N)$, in particular its complex of free factors and its complex of partial bases.
And similarly for $\Out(W_n)$.
Our proof of Theorem~\ref{theointro:bp} relies on techniques developed by Quillen to study the topology of partially ordered sets \cite{Qui}, and is inspired by the work of Brück and Gupta regarding the homotopy type of the free factor complex in the $\Out(F_N)$-setting \cite{BG}.

The analogue of Theorem~\ref{theointro:main} for $\Out(F_N)$ is unknown. Indeed, the homotopy type of the complex of partial bases of the free group is still unknown, and our methods do not seem to apply there. This was conjectured by Day and Putman to be homotopy equivalent to a bouquet of spheres of dimension $N-1$ \cite[Conjecture~1.1]{DP}; this is only known to be true for the $\Aut(F_N)$-version of the complex of partial bases, where elements are not considered up to conjugation, by work of Sadofschi Costa \cite{SC}.

Notice that it is important for us to work with the complex of partial $W_2$-bases since for the free factor analogue, some cell stabilizers $G_\sigma$ have an action on their relative Outer space for which the stabilizers do not all satisfy the cheap rebuilding property.

\paragraph*{Acknowledgments.} D.G.\ was partially supported by the LABEX MILYON (ANR-10-LABX-0070) of Université de Lyon, within the program “Investissements d’Avenir” (ANR-11-IDEX-0007) operated by the French National Research Agency (ANR).
 C.H.\ acknowledges support from the Agence Nationale de la Recherche under Grant ANR22-ERCS-0011-01 Artin-Out-ME-OA, and from the European Research Council under Grant 101040507 Artin-Out-ME-OA. All three authors thank the Institut Henri Poincaré (UAR 839 CNRS-Sorbonne Université) for its hospitality and support (through LabEx CARMIN, ANR-10-LABX-59-01) during the trimester program \emph{Groups acting on fractals, Hyperbolicity and Self-similarity}.

\section{Preliminaries}

\subsection{The cheap rebuilding property and homology growth}

In this section, we review work of Abért, Bergeron, Fr{\k{a}}czyk and the first named author \cite{ABFG} which is crucial for the present paper. They introduced the notion of the \emph{cheap $\alpha$-rebuilding property} (with $\alpha\in\mathbb{N}$) for a countable group $\Gamma$. Before saying more about this property, its importance comes from the following application.

\begin{theo}[{Abért-Bergeron-Fr{\k{a}}czyk-Gaboriau \cite[Theorem~10.20]{ABFG}}]
\label{th:ABFG-10.20}
Let $\Gamma$ be a residually finite countable group. Let $\alpha\in\mathbb{N}$, and assume that $\Gamma$ is of type $F_{\alpha+1}$ and has the cheap $\alpha$-rebuilding property.

Then for every Farber sequence $(\Gamma_k)_{k \in \mathbb N}$ of $\Gamma$, every coefficient field $\mathbb{K}$, and every $j\in\{0,\dots,\alpha\}$, one has
\[\lim_{k\to \infty} \frac{\dim_{\mathbb{K}} H_j (\Gamma_k , \mathbb{K})}{[\Gamma : \Gamma_k]} = 0 \quad \mbox{and} \quad \lim_{k\to \infty} \frac{\log | H_j(\Gamma_k,\mathbb Z)_{\rm tors}|}{[\Gamma:\Gamma_k]}=0.\]
\end{theo}

In the present paper, we will derive our main theorem (Theorem~\ref{theointro:main}) by establishing the cheap $\alpha$-rebuilding property for $\Out(W_n)$, with $\alpha=\lfloor\frac{n}{2}\rfloor-1$ (Theorem~\ref{theo:main} below). The definition of the cheap $\alpha$-rebuilding property being quite involved, we refer the reader to \cite[Definition~10.6]{ABFG}. The important point for us is that it can be established in an inductive way by letting $\Gamma$ act on a sufficiently connected simplicial complex, with simpler stabilizers that have themselves the property. The first building block for us is the following.

\begin{prop}[{Abért-Bergeron-Fr{\k{a}}czyk-Gaboriau \cite[Corollary~10.13(3)]{ABFG}}]\label{prop:abelian}
For every $m\in\mathbb{N}\setminus\{0\}$ and every $\alpha\in\mathbb{N}$, the group $\mathbb{Z}^m$ has the cheap $\alpha$-rebuilding property.
\end{prop}

The following invariance property will also be useful.

\begin{prop}[{Abért-Bergeron-Fr{\k{a}}czyk-Gaboriau \cite[Corollary~10.13(1)]{ABFG}}]\label{prop:finite-index}
Let $\Gamma$ be a countable
residually finite group, let $\Gamma^0\subseteq\Gamma$ be a finite-index subgroup, and let $\alpha\in\mathbb{N}$.

Then $\Gamma$ has the cheap $\alpha$-rebuilding property if and only if $\Gamma^0$ has the cheap $\alpha$-rebuilding property.
\end{prop}
It is worth noticing that finite groups do not have the cheap $\alpha$-rebuilding property, for any $\alpha$. Indeed a residually finite group has the cheap $0$-rebuilding property if and only if it is infinite. 

From these initial building blocks, the cheap $\alpha$-rebuilding property can be proved inductively using the following theorem.

\begin{theo}[{Abért-Bergeron-Fr{\k{a}}czyk-Gaboriau \cite[Theorem~10.9]{ABFG}}]
\label{Th: 10.10 of [ABFG]}
Let $\Gamma$ be a residually finite countable group. Let $\Omega$ be a CW-complex on which $\Gamma$ acts cellularly, in such a way that the pointwise stabilizer of any cell $\omega\in\Omega$ coincides with its setwise stabilizer. Let $\alpha\in\mathbb{N}$, and assume that the following conditions hold:
\begin{enumerate}
    \item $\Gamma\backslash\Omega$ has finite $\alpha$-skeleton;
    \item \label{item: Omega alpha-1-connected} $\Omega$ is $(\alpha-1)$-connected;
    \item for every $j\le\alpha$, the $\Gamma$-stabilizer of any cell of $\Omega$ of dimension $j$ has the cheap $(\alpha-j)$-rebuilding property.
\end{enumerate}
Then $\Gamma$ has the cheap $\alpha$-rebuilding property.
\end{theo}

In the case $\alpha=0$, the assumption (2) in the above theorem is just that $\Omega$ is non-empty. 

Our goal in the present paper will be to establish the cheap $\alpha$-rebuilding property for $\Out(W_n)$ by studying an appropriate action of (some finite-index subgroup of) $\Out(W_n)$ on an appropriate simplicial complex, namely the complex $\BP_n^*$ introduced in Section~\ref{sec:bp} below. The value $\alpha=\lfloor \frac{n}{2}\rfloor-1$ we will obtain relies on the connectivity properties of $\BP_n^*$ (see Theorem~\ref{theo:bp}) together with 
the restriction given by the assumption \eqref{item: Omega alpha-1-connected} in Theorem~\ref{Th: 10.10 of [ABFG]}.

\subsection{\texorpdfstring{The Outer space of $W_n$}{The Outer space of Wn}}

The \emph{Outer space} was first introduced by Culler and Vogtmann \cite{CV} in the context of free groups, and later extended to free products by Guirardel and Levitt \cite{GL}.

A \emph{simplicial Grushko $W_n$-tree} is a simplicial minimal $W_n$-tree $S$ with trivial edge stabilizers in which every nontrivial vertex stabilizer is isomorphic to $\mathbb{Z} /2\mathbb{Z}$ (here \emph{minimal} means that $S$ does not contain any proper nonempty $W_n$-invariant subtree). A \emph{metric Grushko $W_n$-tree} is the metric space obtained from a simplicial Grushko $W_n$-tree $S$ by making every edge of $S$ isometric to an interval of $\mathbb{R}$ of positive length (with edge lengths being $W_n$-invariant), and equipping $S$ with the induced path metric.

Given a metric Grushko $W_n$-tree $T$ and a subgroup $A\subseteq W_n$, there is a unique minimal nonempty $A$-invariant subtree in $T$, which we denote by $T_A$.

The \emph{unprojectivized Outer space} $\mathcal{O}$ of $W_n$ is the space of all equivariant isometry classes of metric Grushko $W_n$-trees. The \emph{projectivized Outer space} $\mathbb{P}\mathcal{O}$ is the space of all equivariant homothety classes of metric Grushko $W_n$-trees (i.e.\ two trees $T,T'$ are identified in $\mathbb{P}\mathcal{O}$ if there exists a $W_n$-equivariant map $T\to T'$ that multiplies all distances by the same positive factor).

The projectivized Outer space $\mathbb{P}\mathcal{O}$ has naturally the structure of a simplicial complex with missing faces, where open simplices are obtained by considering all projective metric trees obtained from a given simplicial Grushko $W_n$-tree by varying edge lengths, keeping them positive. Faces of a simplex are obtained by allowing some edge lengths to degenerate to $0$. 
The missing faces correspond to degeneracies that merge together some vertices with stabilizer $\mathbb{Z} /2\mathbb{Z}$.
As all vertex stabilizers in Grushko $W_n$-trees are finite, the simplicial structure on $\mathbb{P}\mathcal{O}$ is locally finite. The preimage of a simplex of $\mathbb{P}\mathcal{O}$ in $\mathcal{O}$ is a cone.

The space $\mathbb{P}\mathcal{O}$ has several natural topologies described below, which turn out to all be equivalent because the simplicial structure on $\mathbb{P}\mathcal{O}$ is locally finite, see e.g.\ \cite[Proposition~5.4]{GL2}. We will use two of them. First, the \emph{weak topology} is the one where a subset is open if its intersection with every open simplex is open. Second, the \emph{equivariant Gromov-Hausdorff topology} on $\mathcal{O}$ was introduced by Paulin \cite{Pau} as the topology for which a basis of open neighborhoods of a $W_n$-tree $T$ is given by the sets $V_T(K,F,\varepsilon)$, where $K\subseteq T$ and $F\subseteq W_n$ are finite subsets, and $\varepsilon>0$, defined as follows: $V_T(K,F,\varepsilon)$ consists of all trees $T'$ for which there exists a map $\theta:K\to T'$ such that for all $x,y\in K$ and every $g\in F$, one has $|d_{T'}(\theta(x),g\theta(y))-d_T(x,gy)|<\varepsilon$. Informally, this means that two trees are close if any finite subtree $\tau$ in one can be approximated by a finite subtree $\tau'$ in the other, in such a way that the actions of large finite subsets of $W_n$ on $\tau$ and $\tau'$ almost coincide through this approximation. The space $\mathbb{P}\mathcal{O}$ is then equipped with the quotient topology, which is also called the equivariant Gromov-Hausdorff topology. Another commonly used topology (though we will not need it in the present paper) is the \emph{axes topology}, which was shown by Paulin to coincide with the equivariant Gromov-Hausdorff topology \cite{Pau2}. Equipped with any of these topologies, the spaces $\calo$ and $\mathbb{P}\calo$ are contractible \cite{CV,GL}.

A \emph{morphism} between two metric Grushko $W_n$-trees $T,T'$ is a $W_n$-equivariant map $f:T\to T'$ such that every segment of $T$ can be written as a union of finitely many subsegments, each of which is sent isometrically to $T'$.  There is also an equivariant Gromov-Hausdorff topology on the set $\mathcal{M}$ of morphisms between trees in $\calo$, which was introduced by Guirardel and Levitt in \cite[Section~3.2]{GL}. We will only use it in our proof of Lemma~\ref{lem:visible-contractible}. In this topology, a basis of open neighborhoods of a morphism $f:S\to T$ is given by the sets $V_f(K,F,\varepsilon)$, where $K\subseteq S$ and $F\subseteq W_n$ are finite subsets, and $\varepsilon>0$, defined as follows: $V_f(K,F,\varepsilon)$ consists of all morphisms $f':S'\to T'$ between metric Grushko $W_n$-trees for which there exists a map $\theta:K\to S'$ such that for all $x,y\in K$ and every $g\in F$, one has \[|d_{S'}(\theta(x),g\theta(y))-d_S(x,gy)|<\varepsilon~\text{~~~and~~~}|d_{T'}(f'(\theta(x)),f'(g\theta(y)))-d_T(f(x),f(gy))|<\varepsilon.\] The source and range maps are continuous.  

The set of all simplicial Grushko $W_n$-trees is equipped with a partial order, where $S\le T$ if $T$ collapses onto $S$, i.e.\ if $S$ is obtained from $T$ by collapsing to a point each connected component of a given $W_n$-invariant subset. The \emph{spine of Outer space} $|K_n|$ is the simplicial complex defined as the geometric realization of the partially ordered set $K_n$ of Grushko $W_n$-trees \cite{CV,GL}. It naturally embeds in $\mathbb{P}\calo$ as the barycentric spine of its simplicial structure with missing faces: a vertex of $|K_n|$ (corresponding to a simplicial tree $S$) is sent to the barycenter of the simplex of $S$ in $\mathbb{P}\calo$. There is also a deformation retraction $\pi:\mathbb{P}\mathcal{O}\to |K_n|$ with the property that $\pi([T])$ always belongs to the open simplex of $[T]$. In particular $|K_n|$ is also contractible, like $\calo$ and $\mathbb{P}\calo$.

 There is also a relative version of Outer space, introduced by Guirardel and Levitt in \cite[Section~4]{GL}. We will only use it in the proof of Proposition~\ref{prop:stabilizers}. A \emph{free factor system} $\calf$ of $W_n$ is a set $\{[A_1],\dots,[A_k]\}$ of conjugacy classes of subgroups of $W_n$ such that there exist representatives $A_1,\dots,A_k$ and a subgroup $B\subseteq W_n$ such that $W_n=A_1\ast\dots\ast A_k\ast B$. A \emph{metric Grushko $(W_n,\calf)$-tree} is a minimal metric simplicial $W_n$-tree with trivial arc stabilizers, in which all $A_i$ fix a point $v_i$, and all infinite point stabilizers are conjugate to one of the subgroups $A_i$. The \emph{relative Outer space} $\calo(W_n,\calf)$ is the space of all equivariant isometry classes of metric Grushko $(W_n,\calf)$-trees. We denote by $\mathbb{P}\calo(W_n,\calf)$ its projectivized version, where trees are considered up to equivariant homothety instead of isometry. Again $\mathbb{P}\calo(W_n,\calf)$ has a natural structure of a simplicial complex with missing faces. When equipped with the weak topology, it retracts by deformation to its \emph{spine} $|K(W_n,\calf)|$, defined as above using Grushko $(W_n,\calf)$-trees instead of Grushko $W_n$-trees. The space $\mathbb{P}\calo(W_n,\calf)$, equipped with the weak topology, is contractible \cite[Corollary~4.4]{GL}, and therefore so is $|K(W_n,\calf)|$. Although we will not use this fact, let us mention that $\mathbb{P}\calo$ is also contractible in the equivariant Gromov--Hausdorff topology \cite[Theorem~4.2]{GL}: this topology is still equivalent to the axes topology, but is no longer equivalent to the weak topology when some factors in $\calf$ are infinite.

\section{\texorpdfstring{Homotopy type of the complex of partial $W_2$-bases}{Homotopy type of the complex of partial W2 bases}}\label{sec:bp}

We recall that a subgroup $A\subseteq W_n$ is a \emph{free factor} if there exists a subgroup $B\subseteq W_n$ such that $W_n=A\ast B$. A \emph{free $W_2$-factor} is a free factor of $W_n$ isomorphic to $W_2$. A \emph{partial $W_2$-basis} is a free factor system consisting of conjugacy classes of free $W_2$-factors.

The set $\BP_n$ of all partial  $W_2$-bases of $W_n$ is ordered by inclusion. The \emph{complex of partial $W_2$-bases}, denoted by $|\BP_n|$, is the geometric realization of this partially ordered set. It has dimension 
$\lfloor \frac{n}{2}\rfloor-1$. The group $\Out(W_n)$ acts on $|\BP_n|$ by simplicial automorphisms.

We denote by $\Out^0(W_n)$ the finite-index subgroup of $\Out(W_n)$ consisting of all automorphisms that preserve the conjugacy class of each element of order $2$ (as opposed to permuting them). 

We now choose once and for all a finite $n$-tuple  $X^*=(x_1^*,\dots,x_n^*)$ of order $2$ elements forming a basis of $W_n$. We associate them pairwise; more precisely, we say that a free $W_2$-factor of $W_n$ is \emph{$X^*$-paired} if it is of the form $\langle gx_{2i-1}^*g^{-1},hx_{2i}^*h^{-1}\rangle$ for some $g,h\in W_n$ and some $i\in\{1,\dots,\lfloor\frac{n}{2}\rfloor\}$. 
We define $\BP_n^*$ as the subset of $\BP_n$ whose elements are the partial $W_2$-bases consisting only of conjugacy classes of $X^*$-paired free $W_2$-factors, equipped with the subposet structure. Its geometric realization $|\BP_n^*|$ (a subcomplex of $|\BP_n|$) has dimension $\lfloor \frac{n}{2}\rfloor-1$. Notice that the finite-index subgroup $\Out^0(W_n)$ preserves the subcomplex $|\BP_n^*|$. This action is not free. The stabilizers shall be explained later on in Section~\ref{sect:CRP for stab and concl}.

\begin{theo}\label{theo:bp}
For every $n\ge 4$, the complex $|\BP_n^*|$ 
is $\left(\lfloor \frac{n}{2}\rfloor-2\right)$-connected.
\end{theo}

\begin{Remk}\label{Rem: Hurewicz and Whitehead}
Using classical theorems of Hurewicz and Whitehead (as recorded e.g.\ in \cite[Remark~2.9]{BG}), it follows since $|\BP_n^*|$ has dimension $\lfloor \frac{n}{2}\rfloor-1$ that $|\BP_n^*|$ is either homotopy equivalent to a nontrivial bouquet of spheres of dimension $\lfloor \frac{n}{2}\rfloor-1$ or else is contractible. We will see later on (Theorem~\ref{Th. homotopy type} below) that for $n\ge 3$, the complex $|\BP_n^*|$ is indeed noncontractible. For $n=3$, the complex $|\BP^*_3|$ consists of a countable set of isolated points, including all conjugacy classes of free $W_2$-factors of the form $\langle x_1^*, x_3^*(x_1^*x_3^*)^nx_2^*(x_3^*x_1^*)^nx_3^*\rangle$.
\end{Remk}

\begin{Remk}
Contrarily to $|\BP_n^*|$, we expect the complex $|\BP_n|$ not to be $\left(\lfloor \frac{n}{2}\rfloor-2\right)$-connected in general. For instance $|\BP_4|$ is not connected: it has exactly three connected components corresponding to pairing basis elements of $W_n$ according to $X_1^*=(x_1^*,x_2^*,x_3^*,x_4^*)$, to $X_2^*=(x_1^*,x_3^*,x_2^*,x_4^*)$, or to $X_3^*=(x_1^*,x_4^*,x_2^*,x_3^*)$.
This led us to consider $|\BP_n^*|$ instead of $|\BP_n|$. The only place in the proof where the distinction between $|\BP_n|$ and $|\BP_n^*|$ is crucial is Lemma~\ref{lemma:homotopy-type-BP-S}.
\end{Remk}

The general strategy of our proof of Theorem~\ref{theo:bp} is inspired by work of Brück and Gupta \cite[Section~7]{BG}. In order to analyze the homotopy type of $|\BP_n^*|$, we will make use of the following lemmas due to Quillen, applied to well-chosen complexes that will be introduced below.

 Let $Q$ be a poset, we denote its geometric realization by $\vert Q\vert$. 
If $\sigma\in Q$, we consider the following subsets  $Q_{\le \sigma}:= \{\tau\in Q\colon \tau\leq \sigma\}$ and
 $Q_{\ge \sigma} := \{\tau\in Q\colon \tau\ge \sigma\}$.
 
\begin{lemma}[{Quillen \cite[Propositions~1.6 and~7.6]{Qui}}]\label{lemma:quillen}
Let $P,Q$ be posets, and let $f:P\to Q$ be a poset map. 
\begin{enumerate}
\item \label{Quillema contract} Assume that for every $\sigma\in Q$, the geometric realization of the subspace $f^{-1}(Q_{\le \sigma})$ is contractible. Then $f$ induces a homotopy equivalence between $|P|$ and $|Q|$.
\item \label{Quillema m-connect} Let $m\in\mathbb{N}$. Assume that for every $\sigma\in Q$, the geometric realization of $f^{-1}(Q_{\le \sigma})$ is $m$-connected. Then $|P|$ is $m$-connected if and only if $|Q|$ is $m$-connected.
The same conclusion holds if $Q_{\le \sigma}$ is replaced by $Q_{\ge \sigma}$ in the assumption.
\end{enumerate}
\end{lemma}

\begin{lemma}[{Quillen \cite[1.3]{Qui}}]\label{lemma:quillen2}
Let $P$ be a poset and let $f \colon P \to P$ be a poset map such that, for every $\sigma \in P$, we have $f(\sigma) \leq \sigma$. Then the map induced by $f$ on $|P|$ is homotopic to the identity.
\end{lemma}

Let us now introduce a property we shall use to define the space to which we apply these lemmas.
Given a Grushko $W_n$-tree $S$ and a free $W_2$-factor $A\subseteq W_n$, we say that $A$ is \emph{visible in $S$} if the minimal $A$-invariant subtree $S_{A}$ of $S$ is such that, for every $g \in W_n$, either $gS_A=S_A$ or $gS_A \cap S_A$ is at most one point. Equivalently, this means that the quotient graph $A\backslash S_{A}$ embeds in $W_n\backslash S $.
Notice that visibility in $S$ only depends on the conjugacy class of $A$. A partial $W_2$-basis $B=\{[A_1],\ldots,[A_k]\}$ is \emph{visible in $S$} if, for every $i \in \{1,\ldots,k\}$, the free $W_2$-factor $A_i$ is visible in $S$. 

Let $\calz\subseteq\BP_n^*\times K_n$ be the subset consisting of all pairs $(B,S)$ with $B$ visible in $S$. We equip $\BP_n^*\times K_n$ with the product poset structure (i.e.\ $(B,S)\le (B',S')$ if and only if $B\le B'$ and $S\le S'$), and $\calz$ with the subposet structure. 

Our proof has two parts. First we show that its geometric realization $|\calz|$ is homotopically equivalent to $|\BP^*_n|$ (Lemma~\ref{lemma:Z-BP}) by showing that the fibers of the first projection are contractible (Lemma~\ref{lem:visible-contractible}).
Then we determine the homotopy type of $|\calz|$ by using the second projection (Lemma~\ref{lemma:homotopy-type-BP-S}). 

 Given a partial $W_2$-basis $B$, the \emph{visibility space} of $B$ in $|K_n|$, denoted by $\calv^K_B$, is the subcomplex of $|K_n|$ spanned by all vertices corresponding to $W_n$-trees where $B$ is visible. 
 When $B\in \BP_n^*$, then $\calv^K_B$ 
 is the geometric realization of the fiber of $B$ under the first projection $\calz\to \BP_n^*$. 
 Likewise, we define $\calv^{\calo}_B$ and $\calv^{\mathbb{P}\calo}_B$ as the subspaces of $\calo$ and $\mathbb{P}\calo$, respectively, consisting of all trees where $B$ is visible (this makes sense as visibility of $B$ in a tree only depends on the underlying simplicial tree and not on a metric). The following lemma is a variation over \cite[Lemma~5.5]{BG} and crucially relies on the proof by Guirardel and Levitt \cite{GL} of the contractibility of $\calo$.

\begin{lemma}\label{lem:visible-contractible}
 Let $B$ be a partial $W_2$-basis of $W_n$. Then $\calv^{\calo}_B,\calv^{\mathbb{P}\calo}_B$ and $\calv^K_B$ are contractible.
\end{lemma}

\begin{proof}
We will first prove the contractibility of $\calv^\calo_B$. This will be done by working in the equivariant Gromov-Hausdorff topology.

Write $B=\{[\langle x_1,x_2\rangle],\dots, [\langle x_{2k-1},x_{2k}\rangle]\}$, and choose $x_{2k+1},\dots,x_n$ so that $\{x_1,\dots,x_n\}$ is a basis of $W_n$. For every $i\in\{1,\dots,k\}$, let $A_i=\langle x_{2i-1},x_{2i}\rangle$.

Let $S_0\in \calv^\calo_B$ be the  (metric) Grushko $W_n$-tree whose quotient graph $G_0:=W_n\backslash S_0$ is the segment of groups with $n$ vertices whose associated groups are successively $\langle x_{1}\rangle, \langle x_{2}\rangle, \dots, \langle x_{n}\rangle$,
with all edges of length equal to $1$.

For every $T$ in the (unprojectivized) Outer space $\calo$, we now define a morphism $\rho_T:S_0(T)\to T$, where $S_0(T)$ is a tree in $\calo$ obtained from $S_0$ by varying edge lengths in $(0,\infty)$. For this, we first pull-back $G_0$ as a tree inside $S_0$.
The equivariance requires us to send the $S_0$-fixed point $v_j$ of $x_j$ to the $T$-fixed point $w_j$ of $x_j$. 
We extend $\rho_T$ by equivariance to all the vertices of $S_0$. Then we extend it to the edges by linearity. 
We finally adjust the lengths of the edges in $G_0$ so that $\rho_T$ becomes an isometry on each edge of $S_0$: this metric $W_n$-tree is the required $S_0(T)$.

Notice that,
for every $i\in\{1,\dots,k\}$, the minimal subtree of $A_i=\langle x_{2i-1}, x_{2i}\rangle$ in $S_0(T)$ is sent isometrically by $\rho_T$ to the minimal subtree of $A_i$ in $T$. Notice also that, when $T$ belongs to the cone of $S_0$ in $\calo$, one has $S_0(T)=T$.

The trees $S_0(T)$ vary continuously with $T$. The family of morphisms $\rho_T$ is also continuous in the equivariant Gromov-Hausdorff topology. Indeed, every vertex $u$ in $S_0$ is the fixed point of some $g_u\in W_n$, thus  
for every pair of vertices $u,v$ in $S_0$ and every element $g\in W_n$, the points $\rho_T(u)$ and $g\rho_T(v)$ are algebraically determined as the fixed points of $g_u$ and $gg_vg^{-1}$ in $T$. Therefore $d_T(\rho_T(u),g\rho_T(v))$ varies continuously with $T$, so $\rho_T$ varies continuously with $T$. 

Let $\calm$ be the space of $W_n$-equivariant morphisms between metric Grushko $W_n$-trees, equipped with the equivariant Gromov-Hausdorff topology. By \cite[Proposition~3.4]{GL}, there are two continuous maps 
$\Phi, \Psi: \calm \times [0,\infty]\to \calm$ such that if 
 $f:T_0\to T_\infty$ is any morphism then
$\Phi(f,s)$ is a morphism $T_0\to T_s$ and
$\Psi(f,s)$ is a morphism $T_s\to T_\infty$ such that 
$\Psi(f,s)\circ\Phi(f,s)=f$, where 
$\Phi(f,0)=\mathrm{id}_{T_0}$, 
$\Phi(f,\infty)=f$ 
while $\psi(f,0)=f$, 
$\Psi(f,\infty)=\mathrm{id}_{T_\infty}$. See \cite[Remark~4.1]{GL} for why the intermediate trees $T_s$ all belong to $\calo$.

This is used to build a continuous map $r:\calo\times [0,+\infty]\to\calo$ which is a retraction by deformation onto the cone of $S_0$, by considering the composition of the following continuous maps (where the second is $\Phi$ and the third is the map sending a morphism to its range): 
\[\begin{array}{ccccccc}
\calo\times [0,+\infty]&\to &\calm\times [0,+\infty] &\to &\calm & \to &\calo\\
(T,s) & \mapsto & (\rho_T,s) & \mapsto  &\Phi(\rho_T,s) & \mapsto &  \mathrm{range}(\Phi(\rho_T,s))=:T_s
\end{array}\]

Notice that $T_s$ varies continuously between $T_0=S_0(T)$ and $T_{+\infty}=T$, so the above defines a retraction by deformation onto the cone of $S_0$.

In order to deduce the contractibility of $\calv^\calo_B$, there remains to prove that $r(\calv^\calo_B\times [0,+\infty])\subseteq\calv^\calo_B$. It suffices to prove that, for every $T\in\calv_B^\calo$ and every $s \in [0,+\infty]$, we have $T_s \in \calv^\calo_B$. Suppose towards a contradiction that there exist $T\in\calv_B^\calo$ and $s \in (0,+\infty)$ such that $T_s \notin \calv^\calo_B$. There exist $i \in \{1,\ldots,k\}$, $g\in W_n$ and an edge $e$ of $(T_s)_{A_i}$ such that $g(T_s)_{A_i}\neq (T_s)_{A_i}$ and $e \subseteq (T_s)_{A_i} \cap g(T_s)_{A_i}$. 
Let $\tilde{e}$ be the intersection of the full preimage of $e$ under the morphism $\Phi(\rho_T,s)\colon S_0(T)\to T_s$ with the line $S_0(T)_{A_i}$. Since $\rho_T$ sends the $A_i$-minimal subtree of $S_0(T)$ bijectively to the $A_i$-minimal subtree of $T$, the morphism $\Phi(\rho_T,s)$ also sends bijectively $S_0(T)_{A_i}$ to $(T_s)_{A_i}$. Thus $\tilde{e}$ is an interval of $S_0(T)_{A_i}$ which is nonempty and not reduced to a point. Likewise, let $\tilde{e}_g$ be the intersection of the full preimage of $e$ with $gS_0(T)_{A_i}$. Using once again that $\rho_T$ sends the $A_i$-minimal subtree of $S_0(T)$ bijectively to the $A_i$-minimal subtree of $T$, we then have $\rho_T(\tilde{e}) \subseteq T_{A_i}$ and $\rho_T(\tilde{e}_g) \subseteq gT_{A_i}$, and they are not reduced to a point.
On the other hand $\rho_T(\tilde{e})=\Psi(\rho_T,s)(e)=\rho_T(\tilde{e}_g)$. This contradicts the visibility of $B$ in $T$.

We have thus proved the contractibility of $\calv^\calo_B$, and we will now deduce the contractibility of $\calv^{\mathbb{P}\calo}_B$. For this, we choose a continuous section $s:\mathbb{P}\calo\to\calo$ of the natural projection $\chi:\calo\to\mathbb{P}\calo$, by choosing the representative of covolume $1$ of each tree. We then have a retraction $\bar{r}:\mathbb{P}\calo\times [0,+\infty]\to\mathbb{P}\calo$ by considering the composition of the following continuous maps:
\[\mathbb{P}\calo\times [0,+\infty]\to \calo\times [0,+\infty]\to \calo\to\mathbb{P}\calo,\] where the first map is given using the section $s$, the second one is the retraction $r$ constructed above, and the third one is the projection $\chi$. This gives a continuous retraction onto the simplex of $\chi(S_0)$ (which is itself contractible). The retraction $\bar r$ preserves the visibility subspace $\calv^{\mathbb{P}\calo}_B$ (i.e.\ $\bar{r}(\calv^{\mathbb{P}\calo}_B\times [0,+\infty])\subseteq\calv^{\mathbb{P}\calo}_B$) since this is true for $r$ and $\calv^\calo_B$, and since visibility only depends on the homothety class of a tree.

We finally prove the contractibility of $\calv^K_B$. For this, recall that there is a natural embedding of $|K_n|$ in $\mathbb{P}\calo$ as its barycentric spine. Observe that under this embedding, the subcomplex $\calv_B^K$ is mapped into $\calv^{\mathbb{P}\calo}_B$. There is also a continuous retraction $\pi:\mathbb{P}\calo\to |K_n|$, with the property that $\pi([T])$ always belongs to the (open) simplex of $[T]$ -- the continuity of $\pi$ is clear in the weak topology, and since the simplicial structure on $\mathbb{P}\calo$ is locally finite, the weak topology and the equivariant Gromov-Hausdorff topology on $\mathbb{P}\calo$ coincide. We deduce a retraction by deformation $\hat{r}:|K_n|\times [0,+\infty]\to |K_n|$ onto a point (corresponding to the simplex of $\chi(S_0)$), obtained as the composition of the following continuous maps:
\[ |K_n|\times [0,+\infty]\to \mathbb{P}\calo\times [0,+\infty]\to\mathbb{P}\calo\to |K_n|,\] where the first map is obtained from the inclusion $|K_n|\to\mathbb{P}\calo$, the second one is the retraction by deformation $\bar{r}$ constructed above, and the third one is the retraction $\pi$. We finally observe that $\hat{r}$ preserves the visibility subspace $\calv^K_B$ (i.e.\ $\hat{r}(\calv^{K}_B\times [0,+\infty])\subseteq\calv^K_B$), using the facts that $\bar{r}$ preserves visibility, that $\pi$ sends open simplices inside themselves and that visibility is independent of the choice of a tree in an open simplex. 
\end{proof}

Recall that $\calz\subseteq\BP_n^*\times K_n$ is the subset consisting of all pairs $(B,S)$ with $B$ visible in $S$, equipped with the product poset structure.

\begin{lemma}\label{lemma:Z-BP}
The simplicial complex $|\calz|$ is homotopy equivalent to $|\BP_n^*|$.
\end{lemma}

\begin{proof}
We will apply Lemma~\ref{lemma:quillen}(\ref{Quillema contract}) to the projection $\pi_1:\calz\to\BP_n^*$. So let us prove that for every partial basis $B\in\BP_n^*$, the geometric realization of \[\pi_1^{-1}((\BP_n^*)_{\ge B})=\{(B',S)\colon B'\geq B \text{ and $B'$ is visible  in } S\}\] is contractible.
We first observe that if $B'$ is visible in $S$ and $B'\ge B$, then $B$ is visible in $S$. Using Lemma~\ref{lemma:quillen2}, the map $(B',S)\mapsto (B ,S)$ yields a retraction from $|\pi_1^{-1}((\BP_n^*)_{\ge B})|$ to $|\pi_1^{-1}(\{B\})|$.
The space $|\pi_1^{-1}(\{B\})|$ is equal to the subcomplex $\calv^K_B$ of $|K_n|$ spanned by all Grushko $W_n$-trees where $B$ is visible. This is contractible in view of Lemma~\ref{lem:visible-contractible}.
\end{proof}

Let us now turn to determine the homotopy type of $|\calz|$.

Recall that the quotient $W_n\backslash S$ of a Grushko $W_n$-tree $S$ is a finite tree with vertex groups of order $1$ or $2$.
A basis of $W_n$ is \emph{adapted} to $S$ if it consists of the nontrivial stabilizers of the vertices of a connected fundamental domain $L\subseteq S$.

\begin{lemma}\label{lemma:twistor}
Let $S\in K_n$, let $X$ be a basis of $W_n$ adapted to $S$. Let $x,y\in X$, and let $a_x,a_y$ be their fixed points in $S$. Let $x=x_1,x_2,\ldots,x_p=y$ be the generators of the successive nontrivial stabilizers of vertices in the segment $[a_x,a_y]$ (all the $x_i$ belong to $X$ because $X$ is adapted to $S$). 

If $A=\langle x,gyg^{-1}\rangle$ for some $g\in W_n$ is a visible
free $W_2$-factor in $S$ then there exist $\epsilon_{1},\ldots,\epsilon_{p} \in \{0,1\}$ such that $g=x_{1}^{\epsilon_{1}}\ldots x_{p}^{\epsilon_{p}}.$ 

In particular, the set of partial $W_2$-bases that are visible in $S$ is finite.
\end{lemma}

\begin{proof}
First note that $[a_x,ga_y]$ is a fundamental domain for the action of $A$ on $S_A$. Since $A$ is visible in $S$, the image of $[a_x,ga_y]$ in $W_n \backslash S$ is injective and is the same as that of $[a_x,a_y]$. Thus, the segment $[a_x,ga_y]$ contains the same number of vertices with nontrivial stabilizers. 
Let $a_1,\dots,a_p$ be the fixed points in $S$ of $x_1,\dots,x_p$, and let $b_{1},\ldots,b_{p}$ be the vertices of $[a_x,ga_y]$ whose images in $W_n \backslash S$ are the same as $a_{1},\ldots,a_{p}$, respectively. Note that $b_1=a_x$ and that $b_{p}=ga_y$. Let $\ell \in \{2,\ldots,{p}\}$ be the greatest integer such that $b_{\ell} \in [a_x,a_y] \cap [a_x ,ga_y]$. If $\ell=p$, then $b_\ell=a_y$, that is $ga_y=a_y$. Hence $g \in \langle x_p \rangle$ and we are done. Otherwise, since $[a_{\ell},a_{\ell+1}]$ and $[a_{\ell},b_{\ell+1}]$ have the same image in the quotient and have fixed point free interiors, there is a nontrivial element of the stabilizer of $a_{\ell}$, i.e.\ $x_{\ell}$, sending one to the other.
Thus, the segment $[a_x,x_{\ell} b_{\ell+1}]=[a_x,a_{\ell +1}]$ is contained in $[a_x,a_y]$. 
Up to replacing $g$ by $x_{\ell}g$, the length of $[a_x,a_y] \cap [a_x,ga_y]$
has increased.
An immediate finite induction concludes the proof.
\end{proof}

\begin{lemma}\label{lemma:visibility-pointwise-setwise}
Let $S\in K_n$. Let $C=\{[A_1],\dots,[A_k]\}$ be any finite set of conjugacy classes of free $W_2$-factors such that
\begin{enumerate}
    \item each $[A_i]$ is visible in $S$;
    \item given any $i\neq j$ and any representatives $A'_i,A'_j$ in the conjugacy classes $[A_i],[A_j]$, one has $A'_i\cap A'_j=\{1\}$.
\end{enumerate}
Then $C$ is a partial $W_2$-basis.
\end{lemma}

\begin{proof}
Let $\{x_1,\dots,x_n\}$ be a free basis of $W_n$ which is adapted to $S$. For every $i\in\{1,\dots,n\}$, let $a_i$ be the fixed point of $x_i$ in $S$. Since the basis is adapted, the finite subtree $L$ of $S$ spanned by the vertices $a_i$ is a fundamental domain for the $W_n$-action on $S$. Since the quotient $W_n\backslash S$ is a tree, we can (and shall) assume that the numbering of $\{x_1,\dots,x_n\}$ has been chosen so that for every $i\in\{1,\dots,n\}$, the convex hull of $\{a_1,\dots,a_i\}$ does not contain $a_j$ for $j>i$. 

For every $j\in\{1,\dots,k\}$, there exist two integers $\alpha(j)\le\beta(j)$ and $g_j\in W_n$ such that $A'_j=\langle x_{\alpha(j)},g_jx_{\beta(j)}g_j^{-1}\rangle$ is a representative of the conjugacy class $[A_j]$. The second assumption of the lemma ensures that the integers $\alpha(1),\beta(1),\dots,\alpha(k),\beta(k)$ are pairwise distinct.

We will now construct a finite set $\{y_1,\dots,y_n\}$ of elements of order $2$ with the following properties: 
\begin{enumerate}
\item[(i)] for every $j\in\{1,\dots,k\}$, we have $A'_j=\langle y_{\alpha(j)},y_{\beta(j)}\rangle$;
\item[(ii)] for every $i\in\{1,\dots,n\}$, we have $x_i\in\langle y_1,\dots,y_i\rangle$.   
\end{enumerate}
The second property will imply that $\{y_1,\dots,y_n\}$ is a free basis of $W_n$, and the lemma follows.

We now construct the elements $y_i$ inductively. We first set $y_1=x_1$. Let now $i\ge 2$, and assume $y_1,\dots,y_{i-1}$ have already been constructed. 
\\
-- If $x_i$ is not conjugate into any of the free factors $A'_j$, we set $y_i=x_i$, and Property~(ii) holds by induction. 
\\Suppose now that there exists $j\in\{1,\dots,k\}$ such that $x_i$ is conjugate into $A'_j$. \\
-- If $x_i=x_{\alpha(j)}$, we set $y_i=x_i$, and Property~(ii) holds by induction. \\
-- If $x_i=x_{\beta(j)}$, we set $y_i=g_jx_{\beta(j)}g_j^{-1}$. By assumption $[A'_j]$ is visible in $S$. Also, by the choice of the numbering of the set $\{x_1,\dots,x_n\}$, the segment $[a_{\alpha(j)},a_{\beta(j)}]$ does not contain any vertex $a_\ell$ with $\ell>\beta(j)$. We can therefore apply Lemma~\ref{lemma:twistor} and deduce that $g_j\in\langle x_1,\dots,x_{\beta(j)-1}\rangle$. By induction we have $g_j\in\langle y_1,\dots,y_{i-1}\rangle$, hence $x_i\in\langle y_1,\dots,y_i\rangle$. So Property~(ii) holds. 

Finally, our construction ensures that for every $j\in\{1,\dots,k\}$, we have $y_{\alpha(j)}=x_{\alpha(j)}$ and $y_{\beta(j)}=g_jx_{\beta(j)}g_j^{-1}$, so $\langle y_{\alpha(j)},y_{\beta(j)}\rangle=A'_j$. This checks Property~(i), and concludes our proof.  
\end{proof}

Recall that the definition of $\BP^*_n$ relies on a choice of a basis $X^*=(x_1^*, \dots, x_n^*)$ of $W_n$. 
Given $S\in K_n$, we denote by $\BP^*_S$ the subposet of $\BP_n^*$ whose elements are the $X^*$-paired partial $W_2$-bases that are visible in $S$.

\begin{lemma}\label{lemma:homotopy-type-BP-S}
For every $S\in K_n$, the complex $|\BP^*_S|$ is either contractible or homotopy equivalent to a wedge of spheres of dimension $\lfloor\frac{n}{2}\rfloor-1$. 
\end{lemma}

\begin{proof}
For every $i\in\{1,\dots,\lfloor \frac{n}{2}\rfloor\}$, we let $X_i$ be the set of all conjugacy classes of free $W_2$-factors that are visible in $S$ and conjugate to a free factor of the form $\langle x_{2i-1}^*,gx_{2i}^*g^{-1}\rangle$ for some $g\in W_n$. Lemma~\ref{lemma:twistor} implies that $X_i$ is finite. We now let $X=\coprod X_i$. Let $P$ be the partially ordered set (by inclusion) of all nonempty subsets of $X$ that contain at most one element in each subset of the form $X_i$. 

It follows from Lemma~\ref{lemma:visibility-pointwise-setwise} that any collection of conjugacy classes of free $W_2$-factors taken from different $X_i$'s is automatically a partial $W_2$-basis.
Then $\BP^*_S$ is isomorphic to $P$. The conclusion thus follows from the independent Lemma~\ref{lemma:homotopy} below.
\end{proof}

\begin{lemma}\label{lemma:homotopy}
Let $X=\coprod X_i$ be a partition of a finite set $X$ into $k$ disjoint nonempty subsets.
Let $P$ be the partially ordered set (ordered by inclusion) consisting of all finite nonempty subsets of $X$ that contain at most one element in each subset of the form $X_i$.

Then $|P|$ is either contractible or homotopy equivalent to a wedge of spheres of dimension $k-1$.
\end{lemma}

\begin{proof}
The proof is by induction on $k$. If $k=1$, then $|P|$ is just a disjoint union of $\card(X_1)$ points. Let now $k\ge 2$, let $X'=\coprod_{i=1}^{k-1}X_i$, and let $P'$ be the corresponding partially ordered set. Let $Y$ be the cone over $|P'|$, i.e. $Y=(|P'|\times [0,1])/{\sim}$, where $\sim$ is the equivalence relation defined by letting $(x,1)\sim (y,1)$ for all $x,y\in |P'|$. Then $|P|$ is homeomorphic to the space obtained by gluing $\card(X_k)$ copies of $Y$ along the subspaces $|P'|\times\{0\}$ via the identity map. Therefore $|P|$ is contractible if $\card(X_k)=1$ or if $|P'|$ is contractible; otherwise by induction $|P'|$ is homotopy equivalent to a wedge of spheres of dimension $k-2$, and therefore $|P|$ is homotopy equivalent to a wedge of $\card(X_k)-1$ spheres of dimension $k-1$. 
\end{proof}

\begin{proof}[Proof of Theorem~\ref{theo:bp}]
By Lemma~\ref{lemma:Z-BP}, the complex $|\BP^*_n|$ is homotopy equivalent to $|\calz|$. We will apply Lemma~\ref{lemma:quillen}(\ref{Quillema m-connect}) to the projection $\pi_2:\calz\to K_n$. As $|K_n|$ is contractible, we only need to check that for every $S\in K_n$, the geometric realization of the fiber $\pi_2^{-1}((K_n)_{\ge S})$ is $(\lfloor\frac{n}{2}\rfloor-2)$-connected. Notice that given any tree $T\in K_n$ that collapses onto $S$ (i.e.\ $T\geq S$), if a partial $W_2$-basis $B$ is visible in $T$, then it is also visible in $S$.
Therefore, the map $(B,T)\mapsto (B,S)$ defines a retraction from $|\pi_2^{-1}((K_n)_{\ge S})|$ to $|\pi_2^{-1}(S)|=|\BP^*_S|$. Using Lemma~\ref{lemma:quillen2}, this retraction is homotopic to the identity. Thus $|\pi_2^{-1}((K_n)_{\ge S})|$ is homotopy equivalent to $|\BP^*_S|$.
 By Lemma~\ref{lemma:homotopy-type-BP-S}, this space is $(\lfloor\frac{n}{2}\rfloor-2)$-connected, which concludes our proof. 
 \end{proof}

\section{Cheap rebuilding property for the stabilizers, and proof of the main results}
\label{sect:CRP for stab and concl}

In the previous section, we determined the homotopy type of $|\BP_n^*|$. In order to apply the criterion given by Theorem~\ref{Th: 10.10 of [ABFG]}, we are left with establishing the cheap rebuilding property for cell stabilizers of the $\Out^0(W_n)$-action on $|\BP_n^*|$, i.e.\ stabilizers in $\Out^0(W_n)$ of partial $W_2$-bases of $W_n$. This is the contents of Proposition~\ref{prop:stabilizers} below. The proof of our main theorem will be completed afterwards.

Given a finite set $\calf$ of conjugacy classes of subgroups of $W_n$, we denote by $\Out(W_n,\calf)$ the subgroup of $\Out(W_n)$ consisting of all automorphisms that preserve $\calf$ (setwise). Cell stabilizers of the $\Out^0(W_n)$-action on $|\BP_n^*|$ are virtually isomorphic to $\Out(W_n,\mathfrak{B})$ for some partial $W_2$-basis $\mathfrak{B}$ of $W_n$.

\begin{prop}\label{prop:stabilizers}
Let $n\ge 3$, and let $\mathfrak{B}=\{[A_1],\dots,[A_\ell]\}$ be a (nonempty) partial $W_2$-basis of $W_n$. Then the group $\Out(W_n,\mathfrak{B})$ satisfies the cheap $\beta$-rebuilding property for every $\beta\in\mathbb{N}$.
\end{prop}

\begin{proof}
The group $G=\Out(W_n,\mathfrak{B})$ acts on the spine $|K|=|K(W_n,\mathfrak{B})|$ of the relative Outer space, and we will apply Theorem~\ref{Th: 10.10 of [ABFG]} to this action. First, notice that $\Out(W_n)$ is residually finite, see e.g.~\cite[Theorem~1.5]{MO} since $W_n$ is residually finite and infinitely-ended. Therefore, its subgroup $\Out(W_n,\mathfrak{B})$ is also residually finite. The space $|K|$ is contractible by \cite[Corollary~4.4]{GL}. The quotient $\Out(W_n,\mathfrak{B})\backslash |K|$ is a finite simplicial complex. 

We will now prove that the $\Out(W_n,\mathfrak{B})$-stabilizer of any simplex $\tau$ of $|K|$ contains a finite-index infinite abelian subgroup. Indeed, the cell $\tau$ is represented by a finite chain $S_0\to\dots\to S_\ell$ of simplicial Grushko $(W_n,\mathfrak{B})$-trees where the arrows represent collapse maps (i.e.\ $S_{i+1}$ is obtained from $S_i$ by collapsing every connected component of a $W_n$-invariant subforest to a point). Let $\Stab_{G}^0(S_0)$ be the finite-index subgroup of $\Stab_{G}(S_0)$ consisting of the outer automorphisms that act trivially on the quotient graph $W_n\backslash S_0$. Then every element of $\Stab_{G}^0(S_0)$ also fixes all collapsed trees $S_1,\dots,S_\ell$. Thus $\Stab_{G}^0(S_0)\subseteq \Stab_{G}(\tau)\subseteq\Stab_{G}(S_0)$, and these are finite-index inclusions. As every infinite vertex stabilizer in $S_0$ is isomorphic to $W_2$,  it follows from results of Levitt \cite[Propositions 2.2 and 3.1]{Lev} that $\Stab_{G}(S_0)$ is virtually isomorphic to $W_2^d$, where $d\ge 1$ is the number of (oriented) edges of $W_n\backslash S_0$ whose origin has an infinite vertex group. This shows that $\Stab_{G}(\tau)$ contains a finite-index infinite abelian subgroup, and therefore it satisfies the $\beta$-cheap rebuilding property for any $\beta$ by Propositions~\ref{prop:abelian} and~\ref{prop:finite-index}. 

Notice that the description of the simplices of $|K|$ shows that in the action of $\Out(W_n,\mathfrak{B})$ on $|K|$, the setwise stabilizer of any cell is equal to its pointwise stabilizer. Therefore, Theorem~\ref{Th: 10.10 of [ABFG]} applies to show that $G$ satisfies the $\beta$-cheap rebuilding property for every $\beta$.   
\end{proof}

\begin{theo}\label{theo:main}
Let $n\ge 3$. Then $\Out(W_n)$ has the cheap $\alpha$-rebuilding property for $\alpha=\lfloor {\frac{n}{2}}\rfloor-1$.
\end{theo}

\begin{proof}
The group $\Out(W_n)$ is residually finite, see e.g.\ \cite[Theorem~1.5]{MO}.

When $n=3$, the theorem holds because $\Out(W_3)$ 
is infinite (this is exactly the cheap $0$-rebuilding property); it is isomorphic to $\mathrm{PGL}(2,\mathbb{Z})$ (see for instance~\cite[Proposition~2.2]{Gue}). We now assume that $n\ge 4$. By Proposition~\ref{prop:finite-index}, it is enough to prove that the finite-index subgroup $\Out^0(W_n)$ (introduced in the previous section) consisting of all automorphisms that preserve the conjugacy class of each element of order $2$  satisfies the cheap $\alpha$-rebuilding property.

We will apply Theorem~\ref{Th: 10.10 of [ABFG]} to the action of $\Out^0(W_n)$ on $|\BP_n^*|$. Every automorphism that preserves a cell $\sigma$ of $|\BP_n^*|$ fixes it pointwise (as the vertices of a simplex represent partial $W_2$-bases of different cardinalities). By the universal property of free products, $\Out(W_n)$ acts transitively on the set of partial $W_2$-bases of a given cardinality, thus the quotient simplicial complex $\Out(W_n)\backslash|\BP_n|$ is finite, and so is $\Out^0(W_n)\backslash|\BP_n^*|$. By Theorem~\ref{theo:bp}, the complex $|\BP_n^*|$ is $(\alpha-1)$-connected ($n\geq 4$). Now, let $\sigma$ be a cell of $|\BP_n^*|$ corresponding to a chain $\mathfrak{B}_1\subseteq\dots\subseteq \mathfrak{B}_k$ of partial $W_2$-bases, with $\mathfrak{B}_k=\{[A_1],\dots,[A_\ell]\}$. Then the stabilizer $G_\sigma$ of $\sigma$ is equal to $\Out^0(W_n,\mathfrak{B}_k):=\Out^0(W_n)\cap\Out(W_n,\mathfrak{B}_k)$: indeed, the definition of $\Out^0(W_n)$ ensures that every outer automorphism in $\Out^0(W_n,\mathfrak{B}_k)$ fixes each of the conjugacy classes $[A_i]$ (as opposed to permuting them) and therefore it fixes the chain $\mathfrak{B}_1\subseteq\dots\subseteq \mathfrak{B}_k$.
By Proposition~\ref{prop:stabilizers}, the group $\Out(W_n,\mathfrak{B}_k)$ satisfies the cheap $\beta$-rebuilding property for every $\beta\in\mathbb{N}$. Therefore, so does its finite-index subgroup $G_\sigma$.
The assumptions of Theorem~\ref{Th: 10.10 of [ABFG]} are thus satisfied. This concludes our proof.    
\end{proof}

Theorem~\ref{theointro:main} from the introduction now follows from Theorems~\ref{theo:main} and~\ref{th:ABFG-10.20}. We now explain how to deduce Corollary~\ref{corintro:main}.

\begin{proof}[Proof of Corollary~\ref{corintro:main}]
The group $\Out(W_n)$ is of type $F_\infty$. Indeed, using \cite[Theorem~7.3.1]{Geo},
this follows from the fact that it acts without cell inversions on the spine $|K_n|$ of Outer space, which is contractible, with finite stabilizers. Therefore, Lück's approximation theorem \cite{Luc} applies.
\end{proof}

\section{\texorpdfstring{Non-vanishing of the top-dimensional $\ell^2$-Betti number}{Non-vanishing of the top-dim L2 number}}\label{sec:top-dimension}

Applying Gaboriau-Noûs' trick \cite{GN}, we will now establish the following theorem.
\begin{theo}\label{theo:gaboriau-nous}
For every $n\ge 3$, one has $\beta^{(2)}_{n-2}(\Out(W_n))>0$.
\end{theo}

\begin{proof}
As was established by Krsti\'c and Vogtmann in \cite[Corollary~10.2]{KV}, the virtual cohomological dimension of $\Out(W_n)$ is equal to $n-2$. In fact $\Out(W_n)$ acts properly and cocompactly on the contractible simplicial complex $|K_n|$, which is of dimension $n-2$. Therefore, combining \cite[Theorem~1.6 and Proposition~3.1]{GN}, it is enough to find a subgroup of $\Out(W_n)$ isomorphic to $F^{n-3}\rtimes F_2$, where $F$ is a finitely generated free group.

Let $\{x_1,\dots,x_n\}$ be a basis for $W_n$. Let $F$ be a finite-index characteristic free subgroup of $W_3$. Consider the semi-direct product $G=F^{n-3}\rtimes\Aut(W_3)$, where $\Aut(W_3)$ acts diagonally on $F^{n-3}$. Then, after identifying $W_3$ with $\langle x_1,x_2,x_3\rangle$, the group $G$ embeds into $\Aut(W_n)$, by sending $((w_4,\dots,w_n),\varphi)$ to the automorphism $\Phi$ sending $x_i$ to $\varphi(x_i)$ for $i\le 3$, and sending $x_i$ to $w_i^{-1}x_iw_i$ for $i\ge 4$. Notice that whenever $\varphi$ has nontrivial image in $\Out(W_3)$, the resulting automorphism $\Phi$ has nontrivial image in $\Out(W_n)$. Now choose a free subgroup $F_2\subseteq\Aut(W_3)$ which embeds through the quotient map $\Aut(W_3)\to\Out(W_3)$: this is done by lifting a rank $2$ free subgroup of $\Out(W_3)$. We thus obtain an embedding of $F^{n-3}\rtimes F_2$ inside $\Out(W_n)$, which completes our proof.
\end{proof}

\begin{theo}\label{Th. homotopy type}
For every $n\ge 3$, the complex $|\BP_n^*|$ has the homotopy type of a nontrivial bouquet of spheres of dimension $\lfloor\frac{n}{2}\rfloor-1$.
\end{theo}

\begin{proof}
By Theorem~\ref{theo:bp} and Remark~\ref{Rem: Hurewicz and Whitehead}, it suffices to prove that $|\BP_n^*|$ is noncontractible. We established in Proposition~\ref{prop:stabilizers} that all cell stabilizers for the $\Out^0(W_n)$-action on $|\BP_n^*|$ have the cheap $\beta$-rebuilding property for every $\beta\in\mathbb{N}$. If $|\BP^*_n|$ were contractible, the argument from the proof of Theorem~\ref{theo:main} would show that $\Out(W_n)$ has the cheap $\beta$-rebuilding property for every $\beta\in\mathbb{N}$, in particular all its $\ell^2$-Betti numbers would vanish. This is forbidden by Theorem~\ref{theo:gaboriau-nous}.
\end{proof}

    \begin{flushleft}
		Damien Gaboriau\\
		ENS de Lyon, CNRS, 
Unité  Mathématiques Pures et Appliquées, 46 Allée d'Italie, 69007 Lyon, France\\
        Université de Lyon, ENS de Lyon, CNRS, UMPA UMR 5669, F-69364, LYON, France
        \\
		\emph{e-mail:~}\texttt{damien.gaboriau@ens-lyon.fr}
		\\[4mm]
	\end{flushleft}

	\begin{flushleft}
		Yassine Guerch\\
		Universit\'e Paris-Saclay, CNRS,  Laboratoire de math\'ematiques d'Orsay, 91405, Orsay, France \\
		\emph{e-mail:~}\texttt{yassine.guerch@universite-paris-saclay.fr}\\[4mm]
	\end{flushleft}

	\begin{flushleft}
		Camille Horbez\\
		Universit\'e Paris-Saclay, CNRS,  Laboratoire de math\'ematiques d'Orsay, 91405, Orsay, France \\
		\emph{e-mail:~}\texttt{camille.horbez@universite-paris-saclay.fr}\\[4mm]
	\end{flushleft}


\footnotesize

\begin{thebibliography}{ABFG21}

\bibitem[ABFG21]{ABFG}
M.~Abért, N.~Bergeron, M.~Fr{\k{a}}czyk, and D.~Gaboriau.
\newblock On homology torsion growth.
\newblock To appear in {\em J. Eur. Math. Soc. (JEMS)}, 2022.

\bibitem[Ati76]{Ati}
M.~Atiyah.
\newblock Elliptic operators, discrete groups and von Neumann algebras.
\newblock In \emph{Colloque ``Analyse et topologie" en l'Honneur de Henri Cartan (Orsay, 1974)}, 43--72. Astérisque, SMF, no. 32-33. Soc. Math. France, Paris, 1976.

\bibitem[BV20]{BV}
M.~Borinsky and K.~Vogtmann.
\newblock The Euler characteristic of $\mathrm{Out}(F_n)$.
\newblock {\em Comment. Math. Helv.}, 95(4):703-748, 2020.

\bibitem[BG20]{BG}
B.~Brück and R.~Gupta.
\newblock Homotopy type of the complex of free factors of a free group.
\newblock {\em Proc. Lond. Math. Soc. (3)}, 121(6):1737--1765, 2020.

\bibitem[CG86]{CG}
J.~Cheeger and M.~Gromov.
\newblock $L_2$-cohomology and group cohomology.
\newblock {\em Topology}, 25(2):189-215, 1986.

\bibitem[CV86]{CV}
M.~Culler and K.~Vogtmann.
\newblock Moduli of graphs and automorphisms of free groups.
\newblock {\em Invent. Math.}, 84(1):91--119, 1986.


\bibitem[DP13]{DP}
M.~Day and A.~Putman.
\newblock The complex of partial bases for $F_n$ and finite generation of the Torelli subgroup of $\mathrm{Aut}(F_n)$.
\newblock {\em Geom. Dedicata}, 164:139--153, 2013.

\bibitem[GN21]{GN}
D.~Gaboriau and C.~Noûs.
\newblock On the top-dimensional $\ell^2$-Betti numbers.
\newblock {\em Ann. Fac. Sci. Toulouse Math. (6)}, 30(5):1121--1137, 2021.


\bibitem[Geo08]{Geo}
R.~Geoghegan.
\newblock {\em Topological methods in group theory}, volume 243 of {\em
  Graduate Texts in Mathematics}.
\newblock Springer, New York, 2008.

\bibitem[Gue20]{Gue}
Y. Guerch.
\newblock Automorphismes du groupe des automorphismes d'un groupe de Coxeter universel.
\newblock {\em arXiv:2002.02223}, 2020.

\bibitem[Gue21]{Gue3}
Y. Guerch.
\newblock Commensurations of the outer automorphism group of a universal Coxeter group.
\newblock {\em arXiv:2101.07101}, to appear in {\em Groups Geom. Dyn.}, 2021.

\bibitem[GL07a]{GL2}
V.~Guirardel and G.~Levitt.
\newblock Deformation spaces of trees.
\newblock {\em Groups Geom. Dyn.}, 1(2):135--181, 2007. 

\bibitem[GL07b]{GL}
V.~Guirardel and G.~Levitt.
\newblock The outer space of a free product.
\newblock {\em Proc. Lond. Math. Soc. (3)}, 94(3):695--714, 2007.

\bibitem[Har86]{Harer}
J.~L. Harer.
\newblock The virtual cohomological dimension of the mapping class group of an
  orientable surface.
\newblock {\em Invent. Math.}, 84(1):157--176, 1986.

\bibitem[KV93]{KV}
S.~Krstić and K.~Vogtmann.
\newblock Equivariant outer space and automorphisms of free-by-finite groups.
\newblock {\em Comment. Math. Helv.},  68(2):216--262, 1993.

\bibitem[Lev05]{Lev}
G.~Levitt.
\newblock Automorphisms of hyperbolic groups and graphs of groups.
\newblock {\em Geom. Dedicata}, 114:49--70, 2005.

\bibitem[Lüc94]{Luc}
W.~Lück. 
\newblock Approximating $L^2$-invariants by their finite-dimensional analogues. 
\newblock {\em Geom. Funct. Anal.}, 4(4):455--481, 1994.

\bibitem[MO10]{MO}
A.~Minasyan and D.~Osin.
\newblock Normal automorphisms of relatively hyperbolic groups.
\newblock {\em Trans. Amer. Math. Soc.}, 362(11):6079--6103, 2010.

\bibitem[Pau88]{Pau}
F.~Paulin.
\newblock Topologie de Gromov équivariante, structures hyperboliques et arbres réels
\newblock {\em Invent. Math.}, 94(1):53--80, 1988. 

\bibitem[Pau89]{Pau2}
F.~Paulin.
\newblock The Gromov topology on $\mathbb{R}$-trees.
\newblock {\em Topology Appl.}, 32(3):197--221, 1989. 

\bibitem[Qui78]{Qui}
D.~Quillen.
\newblock Homotopy properties of the poset of nontrivial $p$-subgroups of a
  group.
\newblock {\em Adv. in Math.}, 28(2):101--128, 1978.

\bibitem[SC20]{SC}
I.~Sadofschi Costa.
\newblock The complex of partial bases of a free group.
\newblock {\em Bull. Lond. Math. Soc.}, 52(1):109--120, 2020.

\bibitem[Sol69]{Solomon-1969}
L.~Solomon.
\newblock The {S}teinberg character of a finite group with {$BN$}-pair.
\newblock In {\em Theory of {F}inite {G}roups ({S}ymposium, {H}arvard {U}niv.,
  {C}ambridge, {M}ass., 1968)}, pages 213--221. Benjamin, New York, 1969.
  
  

\end{thebibliography}
\end{document}